\documentclass[a4paper,12pt]{article}
\usepackage[english]{babel}
\usepackage[T1]{fontenc}
\usepackage[utf8]{inputenc}
\usepackage[obeyspaces]{url}
\usepackage[dvipsnames]{xcolor}

\usepackage{mathtools, amsmath, amsthm, amssymb, mathtools,
  tikz,tikz-cd, subcaption, framed, float, tensor, bm, dsfont,
  floatpag, microtype, fullpage, graphicx}

\usepackage[normalem]{ulem}
\usepackage{thm-restate}

\usepackage[
  bookmarks=false, 
  colorlinks,
  citecolor=blue,
  linkcolor=blue,
  urlcolor=blue,
]{hyperref}
\usepackage[affil-it]{authblk}
\usepackage{cleveref}
\usetikzlibrary{shapes.geometric, backgrounds, calc,
  arrows, automata, fit, positioning, patterns,
  decorations.pathreplacing, tikzmark, arrows.meta}

\newtheorem{theorem}{Theorem}
\newtheorem{conjecture}{Conjecture}
\newtheorem{question}{Question}
\newtheorem{lemma}{Lemma}

\theoremstyle{definition}
\newtheorem{definition}{Definition}

\newcommand\oeis[1]{\href{https://oeis.org/#1}{#1}}

\newcommand{\W}{\mathcal{W}}

\newcommand{\N}{\mathbb{N}}
\newcommand{\Q}{\mathbb{Q}}
\newcommand{\R}{\mathbb{R}}

\newcommand{\s}{\mathcal{S}} 
\def\S{\mathcal{S}}

\let\le\leqslant
\let\ge\geqslant

\author{Sergey Kirgizov}
\affil{\rm LIB, Université de Bourgogne Franche-Comté\protect\\
  B.P. 47 870, 21078 Dijon Cedex France\protect\\
   {\tt E-mail: sergey.kirgizov@u-bourgogne.fr  }
}

\title{$\mathbb{Q}$-bonacci words and numbers}

\begin{document}

\maketitle

\begin{abstract}
  We present a quite curious generalization of multi-step Fibonacci
  numbers.  For any positive rational $q$, we enumerate binary words
  of length $n$ whose maximal factors of the form $0^a1^b$ satisfy $a
  = 0$ or $aq > b$.  When $q$ is an integer we rediscover classical
  multi-step Fibonacci numbers: Fibonacci, Tribonacci, Tetranacci,
  etc.  When $q$ is not an integer, obtained recurrence relations are
  connected to certain restricted integer compositions.  We also
  discuss Gray codes for these words, and a possibly novel
  generalization of the golden ratio.
\end{abstract}


\section{Introduction}

Multi-step generalization of Fibonacci numbers can be traced back to
the works of Miles~\cite{miles} and 14-year old Feinberg~\cite{fei}. A
lot of different studies about these numbers appear after, including
the works of Flores~\cite{flores}, Miller~\cite{miller},
Dubeau~\cite{dubeau} and Wolfram~\cite{wolf}. A bunch of combinatorial
objects are enumerated by these numbers. For instance, the Knuth's
exercise~\cite[p. 286]{knuth3} shows that the set of length $n$ binary
words avoiding $k$ consecutive 1s is enumerated by $k$-bonacci numbers
respecting $a_n = a_{n-1} + a_{n-2} + \cdots + a_{n-k}$, with initial
conditions $a_0 = 1, a_{-1} = 1,$ and $a_j = 0$ for any $j < -1$.

Independently, in two recent papers~\cite{ourfibo, EI}, a new (as far
as we know) kind of restricted binary words enumerated by generalized
Fibonacci numbers was considered. For any $n \in \N$, Baril, Kirgizov
and Vajnovszki~\cite{ourfibo} defined a set $\W_{q,n}$, parameterized
by a positive natural number $q$, as follows:

\begin{definition}
  $\W_{q,n}$ is the set binary words of length $n$ such that for every
  maximal consecutive subword (factor) of the form $0^a1^b$ which
  satisfies $a > 0$ we have $aq > b$, where $x^\ell$ denotes a factor
  of length $\ell$ consisting only of symbols $x$. Figure~\ref{ex}
  presents some examples.
\label{q}
\end{definition}

Eğecioğlu and Iršič deal in~\cite{EI} with a graph whose vertex set
corresponds to the words from $\W_{1,n}$ starting with zero. Two
vertices are adjacent in this graph if and only if the corresponding
words differ at only one position.

In this short paper, we extend the above definition of $\W_{q,n}$ for
the case where $q$ is a positive rational number, provide generating
functions and give a method to construct linear recurrence relation
for the sequence $(|\W_{q,n}|)_{n\ge 0}$ with 0-or-1 coefficients.

\begin{figure}[ht]
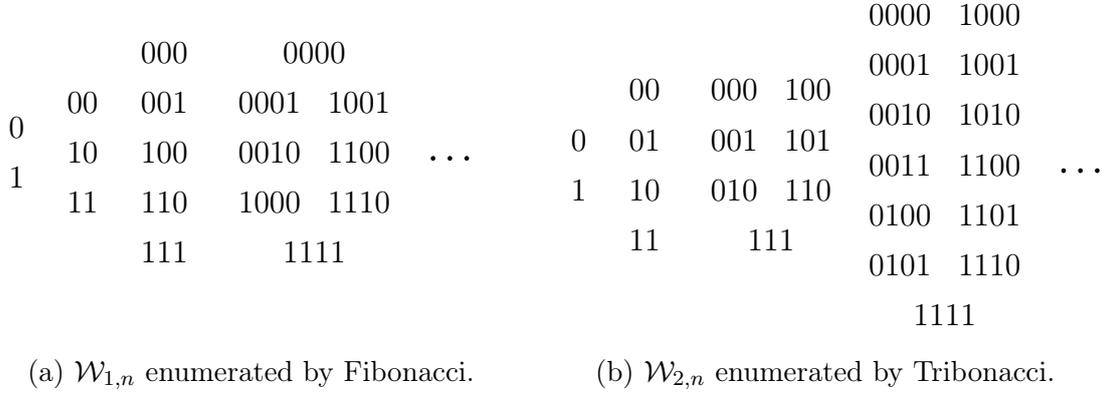

  \centering
  \begin{subfigure}[b]{0.43\textwidth}
      \begin{tabular}{c}
        0 \\
        1
      \end{tabular}
      \begin{tabular}{c}
        00 \\
        10 \\
        11
      \end{tabular}
      \begin{tabular}{c}
        000 \\
        001 \\
        100 \\
        110 \\
        111 
      \end{tabular}
      {\setlength{\tabcolsep}{5pt}
      \begin{tabular}{cc}
        \multicolumn{2}{c}{0000} \\
        0001 & 1001 \\
        0010 & 1100 \\
        1000 & 1110 \\
        \multicolumn{2}{c}{1111} \\
      \end{tabular}
      }
      \vspace*{2em}
          {\Huge ...}
    \caption{$\W_{1,n}$ enumerated by Fibonacci.}\label{ex1}
  \end{subfigure}
  \quad
  \begin{subfigure}[b]{0.45\textwidth}
      \begin{tabular}{c}
        0 \\
        1
      \end{tabular}
      \begin{tabular}{c}
        00 \\
        01 \\
        10 \\
        11
      \end{tabular}
      {\setlength{\tabcolsep}{5pt}
      \begin{tabular}{cc}
        000 & 100\\
        001 & 101 \\
        010 & 110 \\
        \multicolumn{2}{c}{111}
      \end{tabular}
      \begin{tabular}{cc}
        0000 & 1000 \\
        0001 & 1001 \\
        0010 & 1010 \\
        0011 & 1100 \\
        0100 & 1101 \\
        0101 & 1110 \\
        \multicolumn{2}{c}{1111} \\
      \end{tabular}
      }
      {\Huge ...}
    \caption{$\W_{2,n}$ enumerated by Tribonacci.}\label{ex2}
  \end{subfigure}
  \caption{Sets $\W_{q,n}$ for small values of $n$ and $q$.}
  \label{ex}
\end{figure}

\section{Set construction and generating function}

For $q \in \Q^+$, the set $\W_q = \bigcup_{n \in \N} \W_{q,n}$ is constructed as follows:
\begin{equation}
\W_q = \bigcup_{k=0}^{\infty} \{1^k\} \cup \W_q \cdot \s_q, \text{
  where } \s_q = \bigcup_{i=0}^\infty \{ \overbrace{0\ldots00}^{1 +
  \left\lfloor \frac{i}{q} \right\rfloor \text{ zeros}}\mkern-13mu
\underbrace{1\ldots11}_{i \text{ ones}} \;\; \}
\label{struct}
\end{equation}
and $\W_q \cdot \s_q$ corresponds to a set of all possible
concatenations of elements from $\W_q$ and $\s_q$ (in this order).
Table~\ref{t1} shows shortest elements of $\s_q$ for different values
of $q$. A word $111000010000110010 \in \W_{1,18}$ decomposes as $111
\; 0 \; 0 \; 001 \; 0 \; 00011 \; 001 \; 0,$ but a word
$111000010000110010 \in \W_{2,18}$ decomposes as $111 \; 0 \; 0 \; 0
\; 01 \; 0 \; 0 \; 0011 \; 0 \; 01 \; 0$, and $111000010000110010
\notin \W_{1/2}$ because the factor $001$ is not in $\s_{1/2}$ and the
word cannot be constructed.

\begin{table}[ht]
\begin{center}
  \small
  \begin{tabular}{r|r|r|r|r}
$\s_{1/2}$ & $\s_{2/3}$ & $\s_1$ & $\s_2$ & $\s_{3/2}$  \\\hline
0                & 0             & 0           & 0        & 0         \\
0001             & 001           & 001         & 01       & 01        \\
0000011          & 000011        & 00011       & 0011     & 0011      \\
0000000111       & 00000111      & 0000111     & 00111    & 000111    \\
0000000001111    & 00000001111   & 000001111   & 0001111  & 0001111   \\
0000000000011111 & 0000000011111 & 00000011111 & 00011111 & 000011111 \\
$\cdots$ & $\cdots$ & $\cdots$ & $\cdots$ & $\cdots$ \\
  \end{tabular}
\end{center}
\caption{Shortest elements from sets $\s_q$.}
\label{t1}
\end{table}

Let $S_q(x) = \sum_{n = 0}^\infty s_n x^n$ and $W_q(x) = \sum_{n =
  0}^\infty w_n x^n$ be generating functions for $\S_q$ and
$\W_q$, with respect to the word length, marked by $x$. Coefficients
$s_n$ and $w_n$ are the numbers of words of length $n$ from sets
$\S_q$ and $\W_q$. Using the classical symbolic method to derive
formulas for generating functions (see Flajolet-Sedgewick
book~\cite{fla}), we see that $\bigcup_{k=0}^{\infty} \{1^k\}$ has the
generating function $\frac{1}{1-x}$, and Eq.~\eqref{struct} gives
$W_q(x) = \frac{1}{1-x} + W_q (x) S_q(x)$, so
\begin{equation}
  W_q(x) = \frac{1}{\big(1 - S_q(x)\big)(1-x)}.
  \label{wqx}
\end{equation}

In the following we consider a more refined (bivariate) version of
these generating functions with respect to the number of zeros and
ones. We note, with a slight abuse of notation,
\begin{equation}
    W_q(y,z) = \sum_{r = 0}^{\infty} \sum_{i=0}^\infty w_{r,i} z^r y^i,
\end{equation}
where $w_{r,i}$ is the number of words in $\W_q$ having exactly $r$
zeros and $i$ ones. It is easy to see that $\W_q(x)$ is retrieved from
$\W_q(y,z)$ by replacing both $y$ and $z$ by $x$, that is $\W_q(x) =
\W_q (x,x)$. The bivariate generating function $S_q(y,z)$ is defined in a similar
way. In this setting, $\bigcup_{k=0}^{\infty} \{1^k\}$ has the
generating function $\frac{1}{1-y}$, and instead of Eq.~\eqref{wqx} we
have

\begin{equation}
  W_q(y,z) = \frac{1}{\big(1 - S_q(y,z)\big)(1-y)}.
  \label{wqe}
\end{equation}

\bigskip
Now, we construct the set of suffixes $\S_q(y,z)$ and derive its generating
function $S_q(y,z)$.

\begin{definition}
  Let $q = \frac{c}{d}$ be a positive rational number represented by
  the irreducible fraction (e.g. $4 = \frac{4}{1}$), a word factor
  $0^d1^c$ is called a {\em spawning infix}. The generating function
  with respect to the number of zeros (marked by $z$) and the number
  of ones (marked by $y$) for the spawning infix $0^d1^c$ is $z^dy^c$.
  (We intentionally write $z^d$ before $y^c$. According to our idea,
  this should reflect the structure of the factors: zeros appear
  before ones.)
  \label{infix}
\end{definition}

\begin{definition}
  A polynomial
  $$P_{q=\frac{c}{d}}(y,z) = \sum_{i=0}^{c-1} z^{1 + \left\lfloor
    \frac{i}{q} \right\rfloor} y^i $$ is called {\em a model
    polynomial} of a positive rational number $q$ represented by the
  irreducible fraction $q = \frac{c}{d}$.
\end{definition}

For instance, $P_{\frac{2}{3}} (y,z) = z + z^2y$, $P_{\frac{3}{2}}
(y,z) = z + zy + z^2y^2$, and $P_{1/k}(x) = z$ for any $k \in \N^+$.
Figure~\ref{gr} presents a graphical interpretation of model
polynomials.

\begin{figure}[ht]
  \centering
  \includegraphics[width=25em]{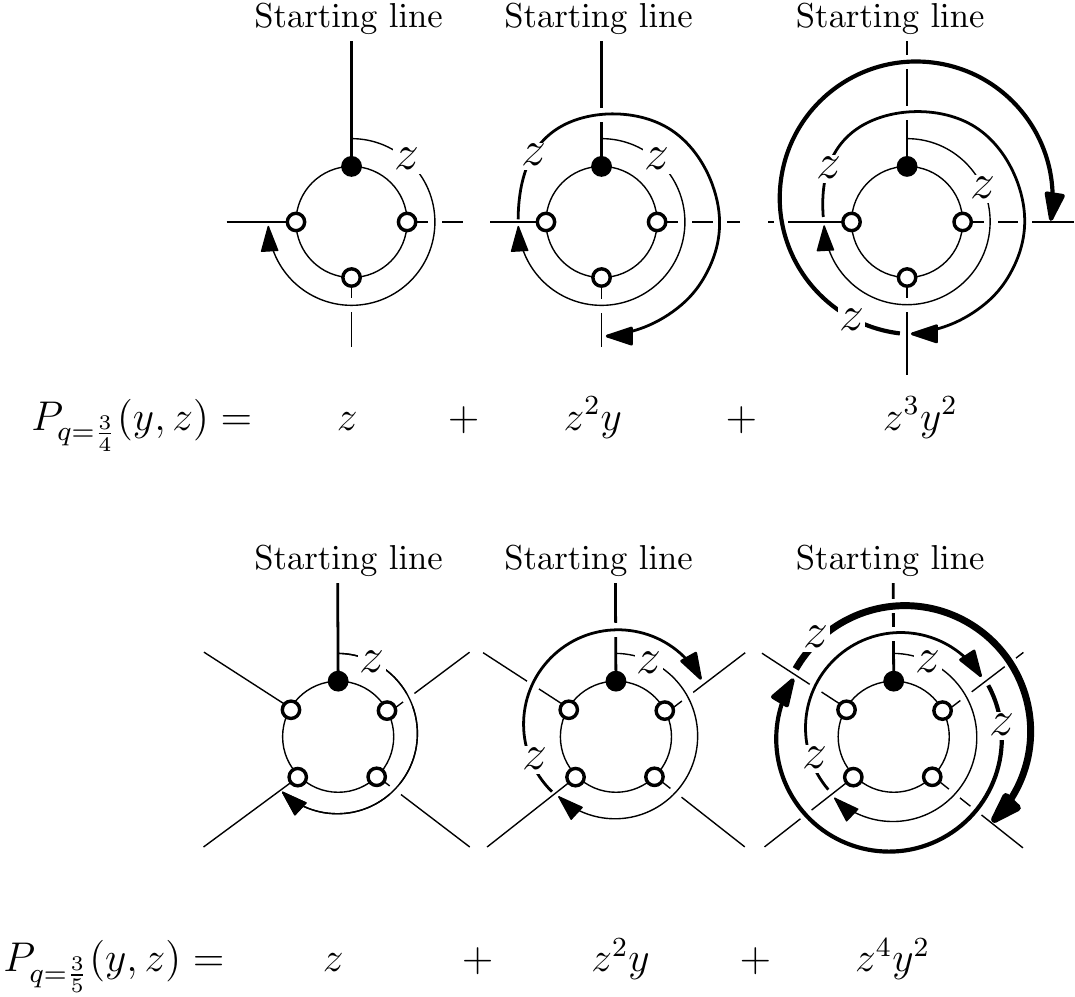}

  \caption{A graphical representation of model polynomial
    $P_{q=\frac{3}{4}} = z + z^2 y + z^3 y^2$. For $j>0$, a term
    $z^iy^j$ in a model polynomial means that one must make $i$
    arc-steps of the angle $2q\pi$ in order to cross the starting line
    $j$ times.  }
  \label{gr}
\end{figure}

\begin{lemma}
  Let $q \in \Q^+$ be represented by the irreducible fraction
  $\frac{c}{d}$.  The generating function $S_q(y,z) =
  \sum_{r=0}^\infty \sum_{i=0}^\infty s_{r,i} z^r y^i$ where $s_{r,i}$
  is the number of words of the form $0^r1^i$,
  where $r = {1+\lfloor i/q \rfloor}$ is
  $$S_{q=\frac{c}{d}}(y,z) = \frac{P_q(y,z)}{1-z^dy^c}.$$
  \label{l1}
\end{lemma}

\begin{proof}
  Let us construct the set $\S_q$ in relation~\eqref{struct}
  iteratively.  First add the word $0$ and all words of the form
  $0^{1+\lfloor i/q \rfloor}1^i$ for $i \in [1,c-1]$.  These words
  correspond to the terms of the model polynomial $P_q(y,z)$. Other
  words of $\S_q$ are obtained by iteratively injecting the spawning
  infix $0^d1^c$ just after the rightmost 0 in already generated
  words.  Using the classical symbolic method~\cite{fla} we see that
  $\frac{1}{{1-z^dy^c}}$ generates a sequence of infix additions. By
  construction $s_{r,i}$ is either 0 or 1.
\end{proof}

To illustrate Lemma~\ref{l1} we take $q=3/2$.  In this case, the model
polynomial is $$P_{\frac{3}{2}} (y,z) = z + zy + z^2y^2, $$
the corresponding words are
$$0, 01, 0011,$$ and the spawning infix is $00111$.  Adding the infix
just after the rightmost 0 we obtain
$$0\underline{00111}, 0\underline{00111}1, 00\underline{00111}11.$$
And repeating this operation, we get
$$000\underline{00111}111, 000\underline{00111}1111, 0000\underline{00111}11111,
00000\underline{00111}111111, \ldots$$
Finally, we obtain the set $\S_{\frac{3}{2}}$.

\bigskip


\begin{theorem}
  The generating function $W_q(y, z) = \sum_{r=0}^\infty
  \sum_{i=0}^\infty w_{r,i} z^r y^i$ where $w_{r,i}$ is number of
  words from $\W_q$ of length $r+i$ containing exactly $r$ zeros and
  $i$ ones is
  $$W_q(y,z) = \frac{1 - z^dy^c}{(1-y)\big(1-z^dy^c - P_q(y,z) \big)}.$$
  \label{wqt}
\end{theorem}

\begin{proof}
  It follows directly from Lemma~\ref{l1} and Equation~\eqref{wqe}.
\end{proof}

Evaluating $W_q(x,x)$ we get the generating function $W_q(x) =
\frac{1 - x^{c+d}}{(1-x)\big(1-x^{c+d} - P_q(x,x) \big)}$
where $x$ marks the length.

The total number of 0s (in other words, the {\em popularity} of 0s) in
all words from $\W_{q=1,n}$ is enumerated by a shift of the sequence
\oeis{A6478} in Sloane's On-line Encyclopedia of Integer
Sequences~\cite{oeis}.  The corresponding generating function is
obtained by evaluating $\frac{\partial W_1(x,xz)}{\partial z}\vert_{z
  = 1}$.  It is quite unexpected, but the sequence \oeis{A6478}
enumerates also the edges in the {\em Fibonacci hypercube} considered
by Rispoli and Cosares~\cite{ris-cos}. A Fibonacci hypercube is a
polytope determined by the convex hull of the {\em Fibonacci cube}
which in turn is defined by Hsu in~\cite{hsu} as the graph whose
vertices correspond to binary words of size $n$ avoiding two
consecutive 1s and where two vertices are connected if and only if the
corresponding words differ at only one position.  Is it possible to
give some kind of a nice bijective construction between the edges of
Fibonacci Hypercube and the 0s in words from $\W_{q=1,n}$?  As far as
we could check, no other
sequences in OEIS~\cite{oeis} correspond to the popularity of 0s (or
1s) for other values of $q$.

\section{Linear recurrence with 0-1 coefficients}

We shall prove the following result.

\begin{theorem}
  Let a positive rational number $q$ be represented by the irreducible
  fraction $\frac{c}{d}$.  The number of $n$-length binary words from
  $\W_{q,n}$, denoted by $w_n$, can be expressed as
  \begin{equation}
    w_n = \sum_{j\in J} w_{n-j} + w_{n-(c+d)},
    \label{req}
  \end{equation}
  where $J$ is the set
  of powers from the model polynomial $P_{q=\frac{c}{d}}(x,x)$. For
  example, when $q=\frac{3}{2}$, we have $P_{\frac{3}{2}}(x,x) = x +
  x^2 + x^4$, and $J = \{1,2,4\}$.
  
  Initial conditions $w_0, w_1, \cdots, w_{c+d-1}$ are obtained by setting $w_n = 0$ for $n < 0$,
  unrolling Equation~\eqref{req} from left to right, while adding an extra 1 for every
  $w_i$ for $0 \le i < c+d$.
  \label{rec}
\end{theorem}

\begin{proof}
  Consider the following map $\psi$ (first defined in~\cite{ourfibo})
  acting on binary words
  \begin{align*}
    \psi(1^k) & = 1^{k+c+d};\\
    \psi(v 1^\ell) & = v 0^d1^c 1^\ell, \text{if } v \text{ ends with } 0.
  \end{align*}

  We first show that $\psi$ induces a bijection from $\W_{q,k}$ to the
  subset of words from $\W_{q,k+c+d}$ ending by at least $c$ 1s.  The
  map $\psi$ inserts the spawning suffix $0^d1^c$ just after the
  rightmost 0 in a word having at least one 0. This does not change
  the property characterizing the words in $\W_q$ (see
  Definition~\ref{q}). If there are no 0s in a word from $\W_{q,k}$,
  this word is extended by adding $c+d$ 1s at the end. And again it
  does not change the characterizing property of $\W_q$. Given the
  above analysis, it easy to see that $\psi$ applied to any word in
  $\W_{q,n}$ gives us a word in $\W_{q,n+c+d}$ and this application is
  bijective.

  As follows from~Equation~\eqref{struct}, any word from $\W_{q,n}$ is
  either $1^n$ or have a form $ps$, where $s = 0^{1 + \lfloor i/q
    \rfloor}1^i$ is a word in $\S_q$ for certain $i \ge 0$, such that
  $n \ge 1 + \lfloor i/q \rfloor + i$ and $p \in \W_{q,n-(1 + \lfloor
    i/q \rfloor + i)}$.  When $n\ge c+d$ there are $c+1$ cases:

  {\bf (case $1$)} The words of $\W_{q,n}$ ending with 0 are obtained
  by adding $0$ at the right end of words from $\W_{q,n-1}$. This
  corresponds to the first term, $z$, of the model polynomial
  $P_{q=\frac{c}{d}}(y,z) = \sum_{i=0}^{c-1} z^{1 + \left\lfloor i/q \right\rfloor}
  y^i$.
  
  {\bf (case $k$, $1 < k < c$)} The words of $\W_{q,n}$ ending with
  $k$ 1s are obtained by adding the suffix $0^{1 + \lfloor k/q
    \rfloor}1^k$ at the right end of words from $\W_{q,n-(1 + \lfloor
    k/q \rfloor + k)}$.  This corresponds to the term $z^{1+\lfloor
    k/q \rfloor}y^k$ of the model polynomial $P_q(y,z)$.
  
  {\bf (case $c+1$)} The words of $\W_{q,n}$ ending with at least $c$
  1s are obtained from the words of $\W_{q,n-(c+d)}$ by applying
  $\psi$.
  
  Considering cardinalities of the sets, these $c+1$ cases give us the
  claimed recurrence relation~\eqref{req}. To construct initial
  conditions $\W_{q,0}, \W_{q,1}, \W_{q,2}, \ldots \W_{q,c+d-1}$, we
  use the same process as in previously considered cases, assuming
  that $\W_{q,m}$ contains no words for every $m < 0$, and adding an
  extra word $1^k$ into every set $\W_{q,k}$ with $0 \le k < c+d$, so
  $\W_{q,0}$ contains only the empty word $1^0$.
\end{proof}

Table~\ref{tab} presents some sequences. Remark, that recurrence
relations for sequences $(|\W_{q,n}|)_{n\ge 0}$ are equal to the
recurrence relations for certain restricted integer compositions
(ordered partitions).  For some values of $q$ the sequence
$(|\W_{q,n}|)_{n\ge 0}$ corresponds exactly to a shift of a sequence
enumerating restricted compositions (see $q=1/5$ in
Table~\ref{tab}). For other values of $q$ the initial conditions
differ from those of integer compositions. Consider, for instance, the
case $q=3/5$.  The recurrence relation is $w_n = w_{n-1} + w_{n-3} +
w_{n-6} + w_{n-8}$.  The same recurrence holds for the sequence
enumerating the compositions of $n \ge 2$ into 1s, 3s, 6s and 8s, but
the initial conditions are different. The sequence of compositions
starts with 1, 2, 3, 4, 7, 11, 17, 27, while the sequence
$(|\W_{3/5,n}|)_{n\ge0}$ begins with 1, 2, 3, 5, 8, 12, 19, 30.

\begin{table}[ht]
  \scriptsize
  \hspace*{-5em}\begin{tabular}{c|l|l|p{9em}} $q$ & Sequence &
    Recurrence relation & OEIS (with shifts)\\ \hline $1/5$ & 1, 2, 3,
    4, 5, 6, 7, 9, 12, 16, 21, 27, ... & $w_n = w_{n-1} + w_{n-6}$ &
    Compositions (ordered partitions) of $n$ into 1s and 6s.
    \oeis{A5708} \\
    
    $1/4$ & 1, 2, 3, 4, 5, 6, 8, 11, 15, 20, 26, 34, ... & $w_n =
    w_{n-1} + w_{n-5}$ & C. into 1s and 5s.  \oeis{A3520} \\
    
    $1/3$ & 1, 2, 3, 4, 5, 7, 10, 14, 19, 26, 36, 50, ... & $w_n =
    w_{n-1} + w_{n-4}$ & C. into 1s and 4s.  \oeis{A3269} \\

    $2/5$ & 1, 2, 3, 4, 6, 9, 13, 18, 26, 38, 55, 79, ... & $w_n =
    w_{n-1} + w_{n-4} + w_{n-7}$ & C. into 1s, 4s and 7s. Not in OEIS.
    \\
    
    $1/2$ & 1, 2, 3, 4, 6, 9, 13, 19, 28, 41, 60, 88, ... & $w_n =
    w_{n-1} + w_{n-3}$ & Narayana's cows, \oeis{A930}\\

    $3/5$ & 1, 2, 3, 5, 8, 12, 19, 30, 46, 72, 113, 176, ... & $w_n =
    w_{n-1} + w_{n-3} + w_{n-6} + w_{n-8}$ & NEW \\
    
    $2/3$ & 1, 2, 3, 5, 8, 12, 19, 30, 47, 74, 116, 182, ... & $w_n =
    w_{n-1} + w_{n-3} + w_{n-5}$ & C. into 1s, 3s and 5s,
    \oeis{A60961} \\

    $3/4$ & 1, 2, 3, 5, 8, 13, 21, 33, 53, 85, 136, 218, ... & $w_n =
    w_{n-1} + w_{n-3} + w_{n-5} + w_{n-7}$ & C. into 1s, 3s, 5s and
    7s, \oeis{A117760} \\

    $4/5$ & 1, 2, 3, 5, 8, 12, 19, 30, 46, 72, 113, 176, ... & $w_n =
    w_{n-1} + w_{n-3} + w_{n-5} + w_{n-7} + w_{n-9}$ & NEW \\

    $1$ & 1, 2, 3, 5, 8, 13, 21, 34, 55, 89, 144, 233, ... & $w_n =
    w_{n-1} + w_{n-2}$ & Fibonacci numbers, \oeis{A45} \\

    $5/4$ & 1, 2, 4, 7, 13, 23, 42, 75, 136, 244, 441, 794, ... & $w_n
    = w_{n-1} + w_{n-2} + w_{n-4} + w_{n-6} + w_{n-8} + w_{n-9} $ &
    NEW \\

    $4/3$ & 1, 2, 4, 7, 13, 23, 42, 75, 136, 245, 443, 799, ... & $w_n
    = w_{n-1} + w_{n-2} + w_{n-4} + w_{n-6} + w_{n-7}$ & NEW \\

    $3/2$ & 1, 2, 4, 7, 13, 23, 42, 76, 138, 250, 453, 821, ... & $w_n
    = w_{n-1} + w_{n-2} + w_{n-4} + w_{n-5} $ & NEW \\

    $5/3$ & 1, 2, 4, 7, 13, 24, 44, 81, 148, 272, 499, 916, ...& $w_n
    = w_{n-1} + w_{n-2} + w_{n-4} + w_{n-5} + w_{n-7} + w_{n-8} $ &
    NEW \\

    $2$ & 1, 2, 4, 7, 13, 24, 44, 81, 149, 274, 504, 927, ... & $w_n =
    w_{n-1} + w_{n-2} + w_{n-3}$ & Tribonacci numbers, \oeis{A73}\\

    $5/2$ & 1, 2, 4, 8, 15, 29, 56, 107, 206, 396, 761, 1463, ...&
    $w_n = w_{n-1} + w_{n-2} + w_{n-3} + w_{n-5} + w_{n-6} + w_{n-7} $
    & NEW \\

    $3$ & 1, 2, 4, 8, 15, 29, 56, 108, 208, 401, 773, 1490, ... & $w_n
    = w_{n-1} + w_{n-2} + w_{n-3} + w_{n-4}$ & Tetranacci numbers,
    \oeis{A78}\\

    $4$ & 1, 2, 4, 8, 16, 31, 61, 120, 236, 464, 912, 1793, ... & $w_n
    = w_{n-1} + w_{n-2} + w_{n-3} + w_{n-4} + w_{n-5}$ & Pentanacci
    numbers, \oeis{A1591}\\

    $5$ & 1, 2, 4, 8, 16, 32, 63, 125, 248, 492, 976, 1936, ... & $w_n
    = w_{n-1} + w_{n-2} + w_{n-3} + w_{n-4} + w_{n-5} + w_{n-6}$ &
    Hexanacci numbers, \oeis{A1592}\\ \hline $\cdots$ & $\cdots$ &
    $\cdots$ & $\cdots$ \\
  \end{tabular}
  \caption{Cardinalities of $\W_{q,n\ge0}$ for some values of $q$.}
  \label{tab}
\end{table}

\section{Gray codes}

A {\em $k$-Gray code}, named after Gray's work~\cite{gray}, for a set
$A$ of words of length $n$ is an arrangement of all words of $A$ in
such a way that any two consecutive words differ at most in $k$
positions.  As follows from a result of~\cite{ourfibo} (which applies
to the rational case also), a 3-Gray code exists for every $\W_{q,n}$
with $n \ge 0$ and any positive rational $q$.

For some values of $q$ and $n$ no 1-Gray code can exist, for example
when $q=2/3$ we have 12 words,
7 with odd number of 1s : $00001, 00100, 00010, 10000, 11001, 11100, 11111$;
and 5 with even number of 1s $00000, 10010, 10001, 11000, 11110$.
It easy to check that there is no 1-Gray in this case.

In general the question whether a 1-Gray code exists for a given $q$
is a challenging one. The Eğecioğlu-Iršič conjecture~\cite{EI} is
essentially about the existence of a 1-Gray code for $W_{1,n}, n \ge
0$. A paper~\cite{ourfibo} offers a proof for this conjecture by
presenting a sophisticated recursive construction. Here is an example
for the words of length 5 and $q=1$: 11111, 11110, 11100, 11000,
11001, 10001, 10000, 10010, 00010, 00011, 00001, 00000, 00100.  As
mentioned in~\cite{ourfibo}, experimental investigations for small
values, $0 \le n \le 5$ and $q \in \{2,3,4,5\}$, suggest the following
conjecture.

\begin{conjecture}[Baril, Kirgizov, Vajnovszki]
  Let $q \in \N^+$ and $n \ge 0$ be given.
  Then, a 1-Gray code exists for $\W_{q,n}$.
\end{conjecture}

\section{Generalized golden ratio}

The generalized golden ratio is defined as $\varphi_k = \lim_{n\to
  \infty}a_{n+1}/a_{n}$, where $a_{n+1}$ and $a_{n}$ are two adjacent
$k$-bonacci numbers.  The golden ratio is $\varphi_2 =
(1+\sqrt{5})/2$, and $\varphi_3 = (1+{\sqrt[{3}]{19+3{\sqrt
      {33}}}}+{\sqrt[{3}]{19-3{\sqrt {33}}}})/3$ is known as the
Tribonacci constant.  The Tetranacci constant $\varphi_4$ have quite a
large expression in radicals. In general, $\varphi_k$ is expressed as
the largest root of the polynomial $x^k - x^{k-1} - \cdots - x - 1$.
See Wolfram's paper~\cite{wolf} for full details.  In the same paper,
Wolfram conjectured that there is no expression in radicals for $k \ge
5$. By computing the Galois group, with the help of the computer
algebra system Magma~\cite{magma}, he confirmed the conjecture for $5
\le k \le 11$.  Martin~\cite{martin} proved the case of even or prime
$k$. Furthermore, Cipu and Luca~\cite{cipu} demonstrated the
impossibility of the construction of $\varphi_k$ by ruler and compass
for $k \ge 3$. As far as we can tell, the question whether there is an
expression in radicals remains open for odd non-prime $k > 11$.
Dubeau~\cite{dubeau} proved that $\varphi_k$ approaches 2 when $k \to
\infty.$

By constructing and enumerating the set $\W_{q,n}$ of restricted
binary words of length $n$, parameterized by a positive rational value
$q$, in this paper we provide a generalization of multi-step Fibonacci
numbers.  For integer $q$ we have $\varphi_{q-1} = \lim_{n\to\infty}
|\W_{q,n+1}|/|\W_{q,n}|$. Non-integer $q$, in some way, allows us to
see what happens with the generalized golden ratio, when its parameter
becomes non-integer.  As the generating functions are rational in our
case, classical analytic combinatorics method can be used to find the
limit.  It equals to $1/\beta$, where $\beta$ the smallest by modulus
root of the denominator of the corresponding generating function
$W_{q=\frac{c}{d}}(x) = \frac{1 - x^{c+d}}{(1-x)\big(1-x^{c+d} -
  P_q(x,x) \big)}$ (see Thm.~\ref{wqt}). Figure~\ref{ratio} presents
some numerical estimations for the function $q \mapsto
\lim_{n\to\infty} |\W_{q,n+1}|/|\W_{q,n}|$, where $q$ takes rational
values from $[0,2.02]$ with step $1/50$.

\begin{question}(related to Wolfram conjecture)
  For which rational values of $q$ there is an expression in radicals
  for $\phi_{q-1}=\lim_{n\to\infty} |\W_{q,n+1}|/|\W_{q,n}|$?
\end{question}

\begin{figure}[h]
  \includegraphics[width=\textwidth]{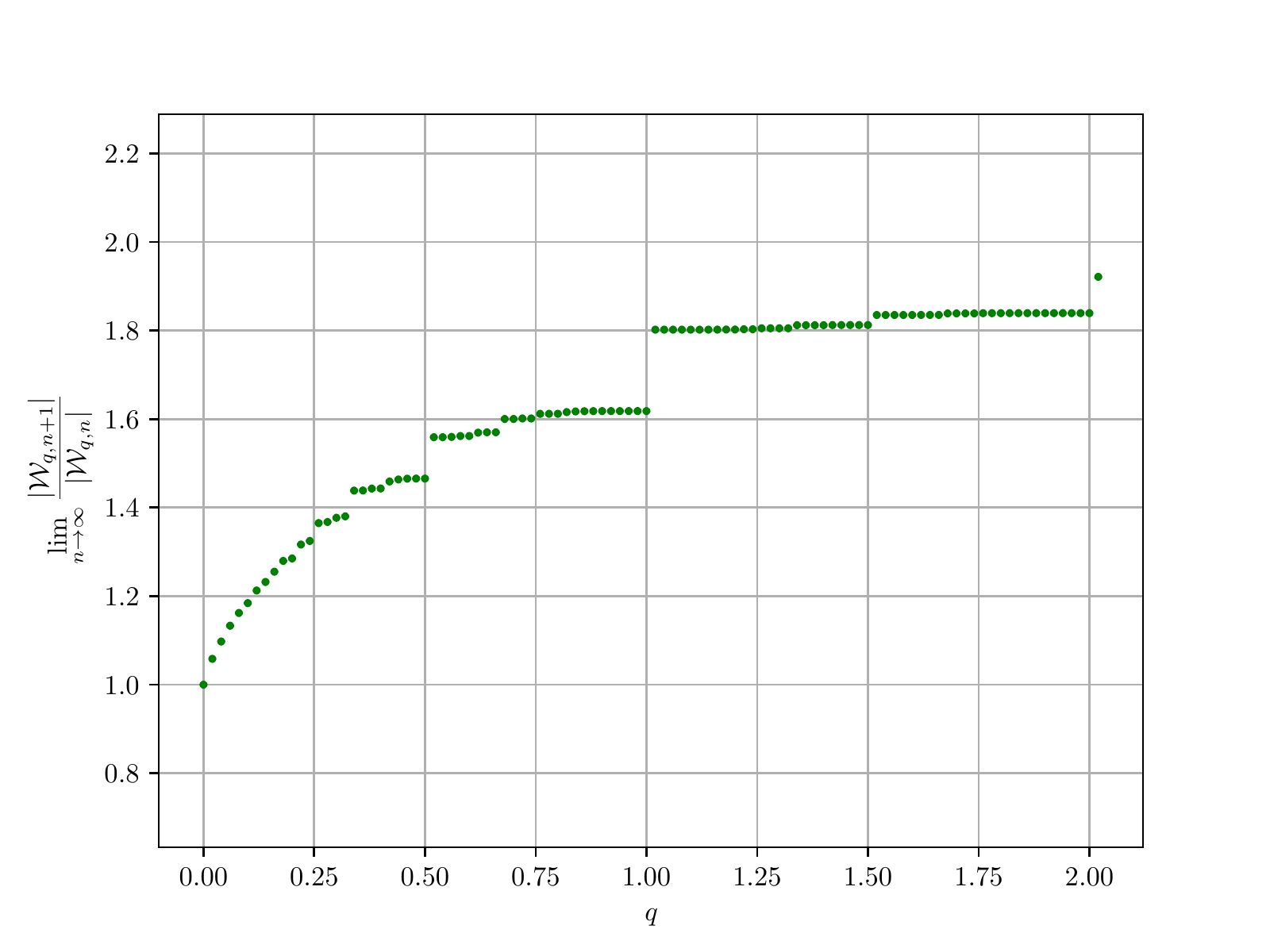}
  \caption{Numerical estimation of $\lim_{n\to\infty}
    |\W_{q,n+1}|/|\W_{q,n}|$ for several values of $r \in [0,2.02]$, using
    a step $0.02$.}
  \label{ratio}
\end{figure}

Remark, that the set $\W_{q,n}$ is well-defined even if we extend the
domain of the parameter $q$ to all positive real numbers. We have two
related conjectures in this realm:

\begin{conjecture}
Let $r \in \R^+$ be given. Then, $\lim_{n\to\infty}
|\W_{r,n+1}|/|\W_{r,n}|$ exists.
\end{conjecture}

\begin{conjecture}
  The function $r \mapsto \lim_{n\to\infty} |\W_{r,n+1}|/|\W_{r,n}|$
  is increasing over the interval $[0,+\infty)$ and discontinuous at
    every positive rational $r$.
\end{conjecture}

\section*{Acknowledgments}

We would like to greatly thank Vincent Vajnovszki and Jean-Luc Baril
for insightful comments and proof-readings, Matteo Cervetti and Rémi
Maréchal for helpful discussions.  This work was supported in part by
the project ANER ARTICO funded by Bourgogne-Franche-Comté region of
France.

\end{document}